\pdfoutput=1
\documentclass[11pt,a4paper]{article}
\usepackage[utf8]{inputenc}

\usepackage{amsfonts}
\usepackage{amsmath}
\usepackage{mathtools}
\usepackage{amsthm}
\usepackage{mathrsfs}
\theoremstyle{definition}
\newtheorem{theorem}{Theorem}[section]
\newtheorem{lemma}[theorem]{Lemma}
\newtheorem{remark}[theorem]{Remark}
\newtheorem{definition}[theorem]{Definition}
\newtheorem{corollary}[theorem]{Corollary}

\title{Finiteness properties and Relatively Hyperbolic Groups}
\date{\vspace{-5ex}}
\author{Harsh Patil}

\begin{document}
\maketitle
\begin{abstract}
    We show that properties $F_n$ and $FP_n$ hold for a relatively hyperbolic group if and only if they hold for all the peripheral subgroups. As an application we show that there are at least countably many distinct quasi-isometry classes of one-ended nonamenable groups that are type $F_n$ but not $F_{n+1}$ and similarly of type $FP_n$ and not $FP_{n+1}$ for all positive integers $n$.  
\end{abstract}
\section{Introduction} 
\footnote{AMS Subject Classification Codes: Primary: 20J05, 20J06, 20F65, 20F67, 20F69}
Relatively hyperbolic groups were introduced by Gromov in \cite{gromov1987hyperbolic} and subsequently expanded upon in \cite{farb1998relatively},\cite{bowditch2012relatively},\cite{osin2006relatively}. Examples include fundamental groups of cusped hyperbolic manifolds (with respect to the collection of cusped subgroups), non-cocompact lattices in Lie groups of real rank 1 and geometrically finite convergence groups acting on non-empty perfect
compact metric spaces (with respect to the set of the maximal
parabolic subgroups \cite{Asli}).
\par
Many geometric and algorithmic properties about a relatively hyperbolic group $G$ can be deduced from the corresponding properties for the peripherals. For example: finiteness of asymptotic dimension \cite{osin2005asymptotic}, solvability of the word problem\cite{farb1998relatively}, solvability of the conjugacy problem\cite{Bumagin2003TheCP}.
\par
 In \cite{Wall}, Wall introduced the property $F_n$ for $n\geq 1$ as a generalization of the classical finiteness properties: finite generation ($F_{1}$) and finite presentability ($F_{2}$). A group is said to be of type type $F_n$ if it acts freely, properly discontinuously, cocompactly, and cellularly on a $(n-1)$-connected CW-complex. A group is said to be of type $F$ if it admits a finite classifying space. 
\par
A group is said to be of type $FP_n$ if the trivial $G$-module $\mathbb{Z}$ admits a projective resolution which is finitely generated in all dimensions up to $n$. For a group of type $F_n$ such a projective resolution can be arrived at by taking the chain complex of the universal cover of some classifying space with finite $n$-skeleton. The first example of a group of type $F_{2}$ but not $F_3$ was exhibited by Stallings \cite{Stallings}. This was generalized by Bieri \cite{Bieri1976NormalSI} to give examples of groups that are $FP_{n}$ but not $FP_{n+1}$ for all $n>1$. Abels-Brown \cite{ABELS198777} constructed examples of solvable groups that satisfy property $FP_{n}$ but not $FP_{n+1}$. A very general construction was then given by Bestvina-Brady\cite{Bestvina1997MorseTA}, who were the first to give examples of groups that are $FP_{n}$ but not $F_{n}$ for each $n\geq2$. 
\par
 The goal of the present paper is to prove the following theorem:
\begin{theorem}\label{main}
Let $G$ be a group and $H<G$ be a subgroup such that $G$ is relatively hyperbolic with respect to $H$. Then,
\begin{enumerate}
    \item \label{part1}$G$ is of type  $FP_n$ if and only if $H$ is of type $FP_n$.
    \item \label{part2}$G$ is of type  $F_n$ if and only if $H$ is of type $F_n$.
\end{enumerate}
\end{theorem}
\begin{remark}
    The analogous statement for more than one peripheral subgroups can be proved similarly. We restrict ourselves to one peripheral subgroup for the sake of clarity. 
\end{remark}
In \cite{osin2006relatively}, Osin proved that a relatively hyperbolic group admits a finite relative presentation. It follows that if all of the peripheral subgroups are finitely presented then the relatively hyperbolic group is also finitely presented. A relatively hyperbolic group quasi-retracts on any peripheral subgroup \cite[Corollary 4.32]{Dahmani2011HyperbolicallyES}. Thus by a well-known result of Alonso \cite{Alonso1994FinitenessCO}, it follows that every peripheral subgroup of a relatively hyperbolic group of type $FP_{n}$ (resp. $F_{n}$) is of type $FP_{n}$ (resp. $F_{n}$).\par
It is known that a finitely presented group is of type $F_{n}$ if and only if it is of type $FP_{n}$ \cite{Brown1987FinitenessPO}. It follows that a relatively hyperbolic group is of type $F_{n}$ if it is of type $FP_{n}$ and all its peripheral subgroups are finitely presented. Thus, to establish Theorem 1.1, it suffices to show that a relatively hyperbolic group is of type $FP_{n}$ if all of its peripheral subgroups are of type $FP_{n}$. This is what we prove here. We achieve this by adapting Dahmani's arguments from \cite{Dahmani}, where it is shown that if a group $G$ is torsion-free and relatively hyperbolic with respect to a subgroup $H$ of type $F$, then $G$ is of type $F$ as well. 
Brown's Finiteness Criterion \cite{Brown1987FinitenessPO} is the main tool that we use to establish property $FP_{n}$. 
We obtain the following result as a corollary to Theorem \ref{main}:
\begin{corollary}\label{humecordes}
     For given $n>0$, there are countably many distinct quasi-isometry types of nonamenable one-ended groups that are of type $FP_n$ ( resp. type $F_{n}$) but not $FP_{n+1}$ (resp. type $F_{n+1}$). 
\end{corollary}
\remark Without the assumption of being one-ended the above result is not hard to prove. Indeed one can take free products $G*H$ where $G$ is chosen to be of type $FP_n$ but not $FP_{n+1}$ and $H$ is a one-ended hyperbolic group. By Theorem 0.4 of \cite{Papasoglu2002QuasiisometriesBG} any two such groups $G*H$ and $G*H'$ are quasi-isometric if and only if $H$ and $H'$ are quasi-isometric. \par
\remark Countably many distinct quasi-isometry classes of amenable groups that are $F_{n}$ (resp. $FP_{n}$) but not $F_{n+1}$ (resp. $FP_{n+1}$) have already been constructed in the literature. Here we outline an example from \cite{SantosRego_2022}. Given a finite set $S$ of primes, one considers the subring $O_{S}$ of $\mathbb{Q}$ that consist of rationals $\frac{a}{b}$, $a,b\in \mathbb{Z}$ relatively prime, such that the prime factors of $b$ are contained in $S$. Let $A_{n}(O_{S})$ denote the group of upper triangular  matrices of size $(n+2)\times (n+2)$ such that the first and the last entry on the diagonal are both equal to 1. By \cite[Proposition 4.10]{SantosRego_2022}, $A_{n}(O_{S})$ is of type $F_{n-2}$ but not $F_{n-1}$ whenever $k=|S|>(n-1)$. By a result of Serre, every finitely generated subgroup of $\textrm{GL}_{n}(\mathbb{Q})$ has finite virtual cohomological dimension. Moreover, $vcd(A_{n}(O_{S}))\geq kn$ since $\mathbb{Z}^{kn}$ embeds as a subgroup in $A_{n}(O_{S})$. Consequently, one can form a sequence of groups that are $F_{n}$ but not $F_{n+1}$ with increasing $vcd$ by choosing larger and larger sets of primes. These groups would be quasi-isometrically distinct as $vcd$ is known to be a quasi-isometry invariant for amenable groups \cite[Theorem 1.5]{Shalom2004HarmonicAC}. 
 \section{Preliminaries}
\subsection{Relative Hyperbolicity}
We recall standard notions from the theory of relatively hyperbolic groups. The interested reader can refer to \cite{farb1998relatively}, \cite{bowditch2012relatively}, \cite{osin2006relatively}.  
\begin{definition}
Let $G$ be a finitely generated group and $H$ be a subgroup of $G$. Let $S$ be some finite generating set of $G$. Let $\Gamma(G,S)$ denote the Cayley graph of $G$ with respect to $S$. If we assign each edge to have unit length then $\Gamma(G,S)$ is a metric space equipped with the path metric. 
\par We form a new graph $\hat{\Gamma}(G,S,H)$ as follows: Start with the Cayley graph $\Gamma(G,S)$ and for each coset $gH$ of $H$ add a vertex $v(gH)$, and add an edge of length $\frac{1}{2}$ between $v(gH)$ and each element $gh$ of $gH$.  
\par This new graph is called the \emph{coned-off Cayley graph} of the pair $(G,H)$ with respect to the generating set $S$.  $\hat{\Gamma}(G,S,H)$ is equipped with the path metric which we denote by $\hat{d}$. 
\end{definition}
\begin{definition}A group $G$ is said to be \emph{weakly relatively hyperbolic} with respect to a subgroup $H$ if the coned-off Cayley graph $\hat{\Gamma}(G,S,H)$ is Gromov hyperbolic. 
\end{definition}
\begin{definition} Let $G$ be a group and $H$ be a subgroup. 
   Let $c:[0,l]\rightarrow \Gamma(G,S)$ be a path in the Cayley graph of $G$. Its image $\hat{c}$ in the coned-off graph $\hat{\Gamma}(G,S,H)$ is constructed by replacing every maximal subpath $c|_{[a,b]}$ all of whose vertices lie in a same coset $gH$ by a path of length $1$ from $c(a)$ to $c(b)$ passing through the cone point $v(gH)$. Let $l'$ denote the length of $\hat{c}$. $c$ is said to be a \emph{relative geodesic} if $\hat{c} $ is a geodesic in $\hat{\Gamma}(G,S,H)$. The length of $\hat{c}$ is said to be the \emph{relative length} of $c$. Let $T>1$. The path $c$ is called a $T$-relative quasi-geodesic if $\hat{c}$ is a $T$-quasi-geodesic i.e.,
   $$T^{-1}|t-t'|\leq \hat{d}(\hat{c}(t),\hat{c}(t'))\leq T|t-t'|\:\:\:\forall t,t'  \in [0,l'] $$
   A path $c$ is said to be a \emph{relative quasi-geodesic} if it is a $T$-relative quasi-geodesic for some $T>0$.
\end{definition}
A path $c:[a,b]\rightarrow \Gamma(G,S)$ is said to \textit{travel in a coset $\gamma H$ for less than $r$} if for any maximal subsegment of $c$ all of whose vertices lie in $\gamma H$, the endpoints are at a distance less than $r$. A path $c$ is said to \emph{travel in a coset $\gamma H$ for more than $r$} if there exists a subsegment of $c$ with vertices in $\gamma H$ such that its endpoints are at least at a distance $r$ from each other. 
\begin{definition}[Bounded coset penetration]
  Let $G$ be a group and $H$ be a subgroup. The pair $(G,H)$ is said to satisfy the BCP (Bounded Coset Penetration) property if for every $T>0$ there exists an $r=r_{BCP}(T)$ such that for every pair of $T$-relative quasigeodesics $c_1$,$c_2$ starting and ending at the same point in $\Gamma(G,S)$ satisfy the following two conditions: 
  \begin{enumerate}
      \item If $c_1$ travels more than $r$ in a coset then $c_2$ enters the same coset.
      \item If $c_1$ and $c_2$ enter the same coset then the entry points (resp. exit points) of $c_1$ and $c_2$ are at most $r$ distance apart. 
  \end{enumerate}
\end{definition}
\begin{definition}[Relative hyperbolicity]
A group is said to be relatively hyperbolic with respect to a subgroup $H$ if it is weakly relatively hyperbolic and satisfies the BCP property.
\end{definition}
\subsection{Finiteness Properties of Groups}

\subsubsection{Brown's Criterion for Finiteness}
In the proof of the “if” part of Theorem \ref{main}, we will make use of Brown’s criteria for finiteness, which we now state for the reader’s convenience.
\begin{definition}
    Let $X$ be a CW-complex which admits a free cellular action by $G$. $X$ is called $n$-\emph{good} if 
    \begin{enumerate}
        \item  $X$ is acyclic in dimensions less than $n$ (i.e. the reduced homology $\Tilde{H}_k(X)$ is $0$ for $ k<n$), and
        \item For $0\leq p<n$, the stabilizer $G_{\sigma}\leq G$ of any $p$-cell $\sigma$ is of type $FP_{n-p}$. 
    \end{enumerate} 
\end{definition}
A \emph{filtration} of $X$ is a family of $G$-invariant sub-complexes $\{X_{\alpha}\}_{\alpha\in D}$ such that $D$ is a directed set and such that $X_{\alpha}\subseteq X_{\beta}$ whenever $\alpha \leq \beta$. When the $X_\alpha$ have a finite $n$-skeleton mod $G$, the filtration will be said to have \emph{finite $n$-type}. 
A filtration $\{X_{\alpha}\}$ of $X$ is said to be $\Tilde{H}_{k}$-\emph{essentially trivial} if for each $\alpha\in D$ there exists a $\beta\in D$ such that $\beta\geq\alpha$ and the inclusion map $i_{\alpha,\beta}:X_{\alpha}\rightarrow X_{\beta}$ induces the trivial map on the homology i.e the map $\tilde{H}_{k}(i_{\alpha,\beta}):\Tilde{H}_{k}(X_{\alpha})\rightarrow \Tilde{H}_{k}(X_{\beta})$ is the zero map. 
We will need the following theorem from \cite{Brown1987FinitenessPO}. 
\begin{theorem}\label{Brown}
Let $X$ be a CW complex which admits a cellular action by $G$. Let $X$ be $n$-good and $\{X_{\alpha}\}$ be a filtration of finite $n$-type. Then $G$ is of type $FP_n$ if and only if the filtration $\{X_{\alpha}\}$ is $\tilde{H}_k$-essentially trivial for each $k<n$.  
\end{theorem}

\subsection{The relative Rips complex}
We will need the following construction from \cite{Dahmani}. Let $G$ be relatively hyperbolic with respect to $H<G$. Let $S$ and $T$ be generating sets of $G$ and $H$, respectively such that $T\subseteq S$. Let $d_T$ denote the word metric on $H$ with respect to $T$.
\begin{definition}\label{relativeRips}
    For $r,d,s>0$ the \emph{relative Rips complex} $Rips_{r,d,s}(G,H)$ is the flag simplicial complex $Y$ defined as follows:
\begin{enumerate}
    \item $Y^{(0)}=G$.
    \item There is an edge between $x,x'\in Y^{(0)}$ if at least one of the following conditions is satisfied: 
    \begin{enumerate}
        \item $x$ and $x'$ lie in the same coset of $H$ and $d_{T}(1,x^{-1}x')\leq s$
        \item there is a relative geodesic $c:[0,l]\rightarrow \Gamma(G,S)$ with endpoints $x,x'$  such that the relative length of $c$ is at most $d$ and it travels less than $3r$ in the first and the last coset and less than $2r$ in any other coset.  
    \end{enumerate}  
   \item A set $\{x_0,x_1,\cdots, x_n\}\subset Y^{(0)}$ spans an $n$-simplex if it forms a complete graph in $Y^{(1)}$.
\end{enumerate}
\end{definition}
\par Let $Rips(G)$ be the infinite simplex on $G$ i.e the simplicial complex with $G$ as its $0$-skeleton and in which every $(n+1)$-element subset of $G$  spans a $n$-simplex.
$Rips(G)$ is contractible and admits a diagonal $G$-action. Moreover, the stabilizer of any simplex is a subgroup of the group of all permutations of the vertices and hence finite. Consequently, $Rips(G)$ is $n$-good for every $n\in \mathbb{N}$.
Varying over different values of $r,d$ and $s$ we obtain the filtration $Rips(G)=\bigcup_{(r,d,s)}Rips_{r,d,s}(G,H)$. The indexing set $\{(r,d,s)|r,d,s>0\}$ is equipped with the lexicographic order. For each $ r,d,s>0$, the complex $Rips_{r,d,s}(G,H)$ is locally finite and finite dimensional. The $G$-orbit of every simplex $\sigma$ in $Rips_{r,d,s}(G,H)$ contains a simplex $\sigma'$ that contains the identity element in its 0-skeleton. It follows that the action of $G$ on  $Rips_{r,d,s}(G,H)$ is cocompact and that the filtration \newline$Rips(G)=\bigcup_{(r,d,s)}Rips_{r,d,s}(G,H)$ is of finite $n$-type. Under the assumptions of Theorem \ref{main} we show in section 3 that the above mentioned filtration is $\Tilde{H}_{k}$-essentially trivial.  \par
Let $x$ and $x'$ be two points lying in distinct cosets of $H$ such that there is an edge $e$ joining $x$ and $x'$ in $Rips_{r,d,s}(G,H)$. By definition, there exists a relative geodesic $c$ such that the condition 2.(b) in \ref{relativeRips} are satisfied. Such a relative geodesic $c$ is said to be a path \textit{associated to the edge} $e$. Let $c:[a,b]\rightarrow \Gamma(G,S)$ be a path such that its image has non-trivial intersection with the coset $\gamma H$. We say that a vertex $v$ of $\Gamma(G,S)$ is the \textit{entering point} of $c$ in the coset $\gamma H$ if $v = c(t) \in \gamma H$ for some $t\in[a,b]$ and $ c(t')\notin \gamma H$ for all $t'<t$.  We say that a path associated to an edge is \textit{reduced} in a coset if its entering point is also its exiting point. In such a case we say that the corresponding edge is reduced with respect to the same coset. The exiting point of $c$ with respect to a coset $\gamma H$ is defined analogously. 
\remark\label{copy} Note that there exists an edge between two elements $gh,gh'$ of a given coset $gH$ if and only if $d_{T}(h,h')\leq s$. Using this observation and the fact that the subcomplex spanned by the coset is flag, one concludes that the subcomplex spanned by each coset $gH$ is canonically isomorphic to the usual Rips complex $Rips_{s}(H)$.
 
\section{Proof of the Main Theorem}

\textit{Idea of the Proof:}
The proof follows the of Lemma 2.2 in \cite{Dahmani} and uses several lemmas from the same. The goal is to show that the filtration of the relative Rips complex outlined in section 2.3 is $\tilde{H}_k$-essentially trivial. Starting with a non-trivial element of $\alpha \in H_{k}(Rips_{r,d,s}(G,H))$ we write it as a sum of elements $\sum\alpha_i$ such that each $\alpha_i$ is supported on a complex spanned by a coset. Once this is achieved we can appeal to the fact that the filtration $ Rips(H)=\bigcup_{r}Rips_{r}(H)$ is  $\tilde{H}_{k}$-essentially trivial and that the subcomplex spanned by each coset is a copy of $Rips_{s}(H)$ . Thus choosing a large enough value $s'>s$ will kill each $\alpha_i$ individually. 
\par
We start by stating a few lemmas from \cite{Dahmani} which will be used later. 
Let $\delta>0$ be the hyperbolicity constant of the coned-off Cayley graph $\hat{\Gamma}(G,S,H)$. 
Let $K$ be an arbitrary finite sub-complex of $Rips_{r,d,s}(G,H)$ for $d>4\delta +2$ and $r>r_{BCP}(4d)$. Fix a base-point $\gamma_b\in K^{(0)}$. Let $\gamma_{0}\in K^{(0)}$ be such that it maximizes the relative distance from $\gamma_b$. Let $\mathcal{F}$ denote the set of all relative geodesics joining $\gamma_{b}$ to a vertex of $K$ which are geodesics in the coset $\gamma_{b}H$. Let $\mathcal{S}_{1}$ be the set of vertices of $\gamma_{b}Rips_{s}(H)$ that either belong to $K$ or to the image of an element of $\mathcal{F}$. Let $S_{2}$ be the set of all $\gamma$ such that an element of $\mathcal{F}$ enters $\gamma H$. Let $\mathcal{S}_{2}$ denote the union of vertices of the subcomplexes $\gamma Rips_{s}(H)$ such that $\gamma\in S_{2}$. Let $\mathcal{V}=\mathcal{S}_{1}\cup \mathcal{S}_{2}$. Let $V$ be the subcomplex spanned by elements in $\mathcal{V}$.   
All subcomplexes of $Rips_{r,d,s}(G,H)$ mentioned below are in fact subcomplexes of $V$. All homotopies between the subcomplexes mentioned below take place inside $V$.
 \begin{lemma}\label{lemma_1}\cite[Lemma 5.1, Corollary 5.1]{Dahmani}
    Let $\gamma_{(0)}^1,\gamma_{(0)}^2\in K^{(0)}$ be vertices of $\gamma_{0}Rips_{s}(H)$. Further assume that $\gamma_{(0)}^1,\gamma_{(0)}^2$ are endpoints of edges $e,e'\in K^{(1)} $  respectively such that $e$ and $e'$ are edges exiting $\gamma_{0} Rips_{s}(H)$ and $e,e'$ are reduced in the coset $\gamma_{0}H$.
    Then $d_{T}(\gamma_{(0)}^1,\gamma_{(0)}^2)\leq r$. Consequently, all such vertices are contained in a simplex $\sigma_{r}$. 
 \end{lemma}
 Here onwards $\sigma_{r}$ denotes the maximal simplex whose vertex set contains the set of endpoints of edges that exit $\gamma_{0} Rips_{s}(H)$ and are reduced in the coset $\gamma_{0}H$. 
 \begin{lemma}\label{lemma_2}\cite[Lemma 5.3, Corollary 5.2]{Dahmani}
    Let $e=(\gamma,\gamma')$ be an edge of $K$ such that $\gamma\in \gamma_{0}H$ and $\gamma'\notin\gamma_{0}H $. Let $\gamma_{r}\in\gamma_{0}H$ be the exiting point from $\gamma_{0}Rips_{s}(H)$ of a path associated to $e$. Let $K_{1}=St_{K}(e)$ be the open star of $e$ in $K$. Let $K_{1}'$ be the union of all simplices of vertices in $K$ that contain the 2-simplex $(\gamma, \gamma',\gamma_{r})$. Let $K'=K\backslash St_{K}(e)$. Then $K_1\subset K_{1}'$. Moreover the inclusion $K'\cap K_{1}'\xhookrightarrow{} K_{1}'$ is a homotopy equivalence.  
 \end{lemma}
\begin{lemma}\label{lemma_3}
     $K$ is homotopic to a sub-complex $\Bar{K}$ of $Rips_{r,d,s}(G,H)$ such that $\bar{K}^{(0)}\subset K^{(0)}$ and the endpoints of all edges exiting $\gamma_{0}Rips_{s}(H)$ in $\Bar{K}$ belong to the simplex $\sigma_{r}$ in $\gamma_{0} Rips_{s}(H)$. 
\end{lemma}
\begin{proof}
     By Lemma \ref{lemma_1}, vertices corresponding to reduced edges are already contained in a simplex. By applying the homotopy given in Lemma \ref{lemma_2} for each edge exiting $\gamma_{0}Rips_{s}(H)$, we obtain a subcomplex $\Bar{K}$ with the desired properties. 
\end{proof}

\begin{lemma}\label{lemma_4}\cite[Lemma 5.4]{Dahmani}
$K$ is homotopic to a finite sub-complex $K''$ such that $K''\cap \gamma_{0} Rips_{s}(H) \subset \sigma_{r}$ and no edge starting at $\sigma_{r}$ is associated to a path which travels more than distance $r$ in a coset.  
\end{lemma}
\begin{lemma}\label{lemma_5}\cite[Lemma 5.7]{Dahmani}
    Let $\tilde{K}$ be a sub-complex such that any edge exiting from $\gamma_{0} Rips_{s}(H)$ starts at a point in a simplex $\sigma$ and is associated to a path traveling no more than $r$ in each coset. Then, there is a  there is a simplicial homotopy of $K''$ that fixes each vertex outside of $\gamma_{0} Rips_{s}(H)$ and sends each vertex in $\gamma_{0} Rips_{s}(H)$ to a vertex $\gamma'$ such that $\hat{d}(\gamma',\gamma_{b})<\hat{d}(\gamma_{0},\gamma_{b})$. 
\end{lemma}
We now have all the tools to prove Theorem \ref{main}. 
\begin{proof}[Proof of Theorem~\ref{main}]:
    Let $d>4\delta+2$, $r>r_{BCP}(4d)$. We show that the filtration $Rips(G)=\bigcup_{(r,d,s)}Rips_{r,d,s}(G,H)$ is $\Tilde{H}_{k}$-essentially trivial. 
Let $\alpha$ be a non-trivial element of $H_{k}(Rips_{r,d,s}(G,H))$ and let $\sigma$ be a $k$-cycle representing $\alpha$. Suppose $\sigma$ is supported on $N$ cosets of $H$. Let $K$ be the sub-complex spanned by the simplices in $\alpha$. One can assume $K$ to be connected without loss of generality. Fix a base-point $\gamma_b \in K^{(0)}$ and let $\gamma_0\in K^{(0)}$ be such that it maximizes the relative distance from $\gamma_{b}$. By using Lemma \ref{lemma_3}, we have that $K$ is homotopic to a complex $K_1$ such that the endpoints of all edges exiting $K_1$ are contained in a simplex $\sigma_{r}$. Let $T$ be the sub-complex obtained by adding the simplex $\sigma_{r}$ and all its faces to $K_1$. $T$ can be expressed as the union of two sub-complexes $A$ and $B$, where $A=T\cap\gamma_{0}Rips_{s}(H)$ and $B=(T \symbol{92} \gamma_{0}Rips_{s}(H) )\cup \sigma_{r}$. Note that, $A\cap B=\sigma_{r}$. The Mayer-Vietoris sequence,
$$\cdots \rightarrow H_{k}(\sigma_{r})\rightarrow H_{k}(A)\oplus H_{k}(B) \rightarrow H_{k}(T) \rightarrow H_{k-1}(\sigma_{r})\rightarrow \cdots $$
 yields an isomorphism  $\tilde{H}_{k}(T)\cong\tilde{H}_{k}(A)\oplus\tilde{H}_{k}(B)$ as $\sigma_{r}$ is a simplex and $H_{i}(\sigma_{r})=\{0\} $ for all $i$.  Thus $\alpha$ can be written as $\alpha=\alpha_1+\alpha'$ such that $\alpha_1\in\tilde{H}_{k}(A)$ and $\alpha'\in\tilde{H}_{k}(B)$. Let $\theta \in Z_{k}(B)$ be a cycle representing $\alpha'$ and let $K_{2}$ be the subcomplex spanned by simplices in $\theta$. By using Lemma \ref{lemma_4} we get a subcomplex $K_3$ homotopic to $K_2$ such that no edges starting at $\sigma_r$ are associated to a path that travels more than distance $r$ in a coset. We are now in a position to use Lemma \ref{lemma_5}. We get a subcomplex $K_4$ homotopic to $K_3$ such $K_{4}$ lies entirely outside of $\gamma_{0}H$. Thus, the number of cosets in the support of $K_{4}$ that maximize the relative distance from $\gamma_{b}$ is strictly less than the number of cosets in the support of $K$ that maximize the relative distance from $\gamma_{b}$. One can repeat these steps each time replacing $\alpha$ by $\alpha'$ and $K$ by the complex $K_{4}$ obtained at the end of the previous step. After finitely many steps the relative diameter of $K$ will become zero in which case the cycle $\alpha'$ will be supported on a single coset.  Thus one can inductively write $\alpha=\alpha_1+\cdots +\alpha_N$ where each $\alpha_i$ is supported on some coset $\gamma_i H$.  
By assumption, $H$ is of type $FP_n$. Hence by Theorem \ref{Brown}, the filtration $Rips(H)=\bigcup_{r} Rips_{r}(H)$ is $\Tilde{H}_k$-essentially trivial. Thus there exists an $s'>s$ such that the inclusion map $i_{s,s'}:Rips_{s}(H)\rightarrow Rips_{s'}(H)$ induces a trivial map on the homology. The subcomplex spanned by each coset in $Rips_{r,d,s}(G,H)$ is isomorphic to $Rips_{s}(H)$ by Remark \ref{copy} . Let $j_{s,s'}:Rips_{r,d,s}(G,H)\rightarrow Rips_{r,d,s'}(G,H)$  be the canonical inclusion map. Then $\Tilde{H}(j_{s,s'})(\alpha)=\sum \Tilde{H}(j_{s,s'})(\alpha_i)=0$. \par

\end{proof} 

 \section{An Application}
In this section we apply Theorem \ref{main} to prove corollary \ref{humecordes}. 
 We will need the following result from \cite{Cordes2016RelativelyHG}. 
 \begin{theorem}\cite[Theorem 1.6]{Cordes2016RelativelyHG}
     Let $\mathcal{H}$ be a collection of groups each of which has finite stable dimension or are non-relatively hyperbolic (in the sense of \cite{Behrstock2005ThickMS}). There is an infinite family of 1-ended groups $\{G_n\}_n$ which are hyperbolic relative to $\mathcal{H}$ such that $G_i$ is not quasi-isometric to $G_j$ for $i\neq j$. 
 \end{theorem}
 \begin{proof}[Proof of Corollary \ref{humecordes}]:
     Let $H$ be a group that is of type $FP_n$ and not of type $FP_{n+1}$. Such a group has been constructed in \cite{Bestvina1997MorseTA}. Moreover $H$ can be chosen such that it is a nonamenable subgroup of a RAAG and thus has finite asymptotic dimension. In particular, it satisfies the hypothesis of Theorem \ref{humecordes}. Let $G_n$ be a sequence of one-ended groups that are relatively hyperbolic with respect to $\mathcal{H}=\{H\} $ such that $G_i$ and $G_j$ are quasi-isometrically distinct for $i\neq j$. Each $G_n$ is of type $FP_n$ but not of type $FP_{n+1}$.  
 \end{proof}
 
\section*{Acknowledgements}
We would like to thank Ian Leary for pointing us in the right direction during a conversation. We would like to thank David Hume and Mark Hagen for several helpful discussions. We would like to thank the anonymous referee for their detailed comments on a previous version of this paper. 

\bibliographystyle{plain}
\bibliography{main}

\textsc{School of Mathematics, University of Bristol, Bristol, BS8 1UG} 
\par\nopagebreak
\textit{E-mail address}: \texttt{cr22307@bristol.ac.uk}
\end{document}